\documentclass[10pt,twoside]{article}
\usepackage{mathrsfs}
\usepackage{amssymb}
\usepackage{amsmath}
\usepackage[all]{xy}
\usepackage{amsthm}
\numberwithin{equation}{section}

\newtheorem{theorem}{Theorem}[section]

\newtheorem{definition}[theorem]{Definition}
\newtheorem{remark}[theorem]{Remark}
\newtheorem{lemma}[theorem]{Lemma}
\newtheorem{example}[theorem]{Example}

\newtheorem{corollary}[theorem]{Corollary}

\newcommand{\edge}{\ar@{-}}

\newcommand{\pf}{\noindent\begin {proof}}
\newcommand{\epf}{\end{proof}}

\newcommand{\Ext}{\mbox{\rm Ext}}
\newcommand{\Hom}{\mbox{\rm Hom}}

\def\Im{\mathop{\rm Im}\nolimits}
\def\Ker{\mathop{\rm Ker}\nolimits}
\def\Coker{\mathop{\rm Coker}\nolimits}

\def\La{\mathop{\rm \Lambda}\nolimits}

\def\mod{\mathop{\rm mod}\nolimits}

\def\id{\mathop{\rm id}\nolimits}
\def\pd{\mathop{\rm pd}\nolimits}

\def\max{\mathop{\rm max}\nolimits}
\def\min{\mathop{\rm min}\nolimits}
\def\sup{\mathop{\rm sup}\nolimits}
\def\inf{\mathop{\rm inf}\nolimits}
\def\add{\mathop{\rm add}\nolimits}

\def\gldim{\mathop{\rm gl.dim}\nolimits}

\def\rad{\mathop{{\rm rad}}\nolimits}
\def\top{\mathop{{\rm top}}\nolimits}

\def\dim{\mathop{\rm dim}\nolimits}

\def\Hom{\mathop{\rm Hom}\nolimits}
\def\Ext{\mathop{\rm Ext}\nolimits}
\def\sup{\mathop{\rm sup}\nolimits}
\def\lim{\mathop{\underrightarrow{\rm lim}}\nolimits}

\def\Ext{\mathop{\rm Ext}\nolimits}

\def\End{\mathop{\rm End}\nolimits}

\def\mod{\mathop{\rm mod}\nolimits}

\def\id{\mathop{\rm id}\nolimits}
\def\pd{\mathop{\rm pd}\nolimits}
\def\max{\mathop{\rm max}\nolimits}
\def\min{\mathop{\rm min}\nolimits}
\def\sup{\mathop{\rm sup}\nolimits}
\def\inf{\mathop{\rm inf}\nolimits}
\def\add{\mathop{\rm add}\nolimits}

\def\gldim{\mathop{\rm gl.dim}\nolimits}

\def\rad{\mathop{{\rm rad}}\nolimits}

\def\top{\mathop{{\rm top}}\nolimits}

\def\dim{\mathop{\rm dim}\nolimits}
\def\extdim{\mathop{\rm ext.dim}\nolimits}
\def\derdim{\mathop{\rm der.dim}\nolimits}

\def\Hom{\mathop{\rm Hom}\nolimits}
\def\Ext{\mathop{\rm Ext}\nolimits}
\def\sup{\mathop{\rm sup}\nolimits}
\def\lim{\mathop{\underrightarrow{\rm lim}}\nolimits}

\def\Ext{\mathop{\rm Ext}\nolimits}

\def\End{\mathop{\rm End}\nolimits}

\def\repdim{\mathop{\rm rep.dim}\nolimits}

\def\A{\mathop{\rm \mathcal{A}}\nolimits}
\def\B{\mathop{\rm \mathcal{B}}\nolimits}
\def\C{\mathop{\rm \mathcal{C}}\nolimits}

\def\V{\mathop{\rm \mathcal{V}}\nolimits}
\def\T{\mathop{\rm \mathcal{T}}\nolimits}
\def\I{\mathop{\rm \mathcal{I}}\nolimits}

\def\X{\mathop{\rm \mathcal{X}}\nolimits}

\def\LL{\mathop{\rm LL}\nolimits}

\title{ \bf The Derived and Extension Dimensions of Abelian Categories
}
\vspace{0.2cm}

\author { \ Junling  Zheng$^a$, Zhaoyong Huang$^{b,}$\thanks{E-mail addresses: zjlshuxue@163.com (J. Zheng), huangzy@nju.edu.cn (Z. Huang)}
\thanks{Corresponding author.}
\\
{\it \scriptsize  $^a$ Department of Mathematics, China Jiliang University, Hangzhou 310018, Zhejiang Province, P. R. China;
}\\{\it \scriptsize $^b$ Department of Mathematics, Nanjing University,
Nanjing 210093, Jiangsu Province, P.R. China }}

\date{ }

\begin{document}

\baselineskip=16pt


\maketitle

\begin{abstract}
For an abelian category $\mathcal{A}$, we establish the relation between
its derived and extension dimensions.
Then for an artin algebra $\Lambda$, we give the upper bounds of
the extension dimension of $\Lambda$ in terms of the radical layer length of $\Lambda$ and certain
relative projective (or injective) dimension of some simple $\Lambda$-modules,
from which some new upper bounds of the derived dimension of $\Lambda$ are induced.
\end{abstract}

\pagestyle{myheadings}
\markboth{\rightline {\scriptsize  J. L. Zheng, Z. Y. Huang\emph{}}}
         {\leftline{\scriptsize  The Derived and Extension Dimensions of Abelian Categories }}


\section{Introduction} 
Given a triangulated category $\mathcal{T}$, Rouquier introduced in
\cite{rouquier2006representation,rouquier2008dimensions}
the dimension $\dim\mathcal{T}$ of $\mathcal{T}$
under the idea of Bondal and van den Bergh in \cite{bondal2003generators}.
This dimension and the infimum of the Orlov spectrum of $\mathcal{T}$
coincide, see \cite{ballard2012orlov,orlov2009remarks}. Roughly speaking,
it is an invariant that measures how quickly the category can be built from one object.
This dimension plays
an important role in representation theory. For example,
it can be used to compute the representation dimension of
artin algebras (\cite{rouquier2006representation,oppermann2009lower}).
Many authors have studied the upper bound of $\dim \mathcal{T}$, see
\cite{ballard2012orlov,bergh2015gorenstein,chen2008algebras,han2009derived,oppermann2012generating,
rouquier2006representation,rouquier2008dimensions,zheng2017upper} and so on.

There are a lot of triangulated categories having infinite dimension;
for instance, Oppermann and \v St'ov\'\i \v cek proved in \cite{oppermann2012generating} that
all proper thick subcategories of the bounded derived category of finitely generated modules over a Noetherian algebra
containing perfect complexes have infinite dimension.
Let $\Lambda$ be an artin algebra and $\mod \Lambda$ the category of finitely generated right $\Lambda$-modules.
It was proved in \cite[Propositions 7.37 and 7.4]{rouquier2008dimensions} that the
dimension of the bounded derived category $D^{b}(\mod \Lambda)$ is at most $\min\{\LL(\Lambda)-1,\gldim\Lambda\}$,
where $\LL(\Lambda)$ and $\gldim\Lambda$ are the Loewy length and global dimension of $\Lambda$ respectively.

As an analogue of the dimension of triangulated categories, the (extension) dimension $\extdim\A$ of an abelian category $\A$
was introduced by Beligiannis in \cite{beligiannis2008some}, also see \cite{dao2014radius}. Let $\Lambda$ be an artin algebra. Note that
the representation dimension of $\Lambda$ is at most two (that is, $\Lambda$ is of finite representation type) if and only if
$\extdim\mod \Lambda=0$ (\cite{beligiannis2008some}). So, like the representation dimension of $\Lambda$, the extension dimension
$\extdim\mod \Lambda$ is also an invariant that measures how far $\Lambda$ is from of finite representation type.
It was proved in \cite{beligiannis2008some,zheng2019extension} that $\extdim\mod \Lambda\leqslant\min\{\LL(\Lambda)-1,\gldim\Lambda\}$,
which is a counterpart of the above result of Roquier.
In \cite{zheng2017upper,zheng2019extension}, we obtained many upper bounds of $\dim D^{b}(\mod \Lambda)$ and $\extdim\mod \Lambda$
in terms of the radical layer length of $\Lambda$ and the projective (or injective) dimension of some simple $\Lambda$-modules,
such that the upper bounds $\LL(\Lambda)-1$ and $\gldim\Lambda$ are special cases.

In this paper, for an abelian category $\mathcal{A}$, we establish the relation between the dimensions of $\mathcal{A}$ and
the bounded derived category $D^{b}(\mathcal{A})$ of $\mathcal{A}$.
Then for an artin algebra $\Lambda$, we give the upper bounds of $\extdim\mod \Lambda$
in terms of the radical layer length of $\Lambda$ and certain relative projective (or injective) dimension of some simple $\Lambda$-modules,
from which some new upper bounds of $\dim D^{b}(\mod \Lambda)$ are induced. The paper is organized as follows.

In Section 2, we give some terminology and some preliminary results.

Let $\mathcal{A}$ be an abelian category. The dimensions of $D^{b}(\mathcal{A})$ and $\mathcal{A}$ are usually called
the derived and extension dimensions of $\mathcal{A}$, and denoted by $\derdim\mathcal{A}$ and $\extdim\mathcal{A}$ respectively.
In Section 3, we get that $\derdim\mathcal{A}\leqslant 2\extdim\mathcal{A}+1$ (Theorem \ref{thm-3.3}).
Let $\Lambda$ be an artin algebra and $\mathcal{V}$ a certain class of simple $\Lambda$-modules. Then for a subcategory $\mathcal{X}$
of $\mod\Lambda$ of finite type, we give an upper bound of
$\extdim\mod\Lambda$ in terms of the $\mathcal{X}$-projective (or $\mathcal{X}$-injective) dimension of $\mathcal{V}$ and the radical layer
length of $\Lambda$ (Theorem \ref{thm-3.12}). Combining this result with Theorem \ref{thm-3.3}, we get
some new upper bounds of $\derdim \mod\Lambda$ (Theorem \ref{thm-3.18}).

In Section 4, we give two examples to illustrate that in some cases, the upper bounds obtained in this paper are more precise,
even arbitrarily smaller, than that in the literature known so far, and that we may obtain the exact value of the derived dimension
of some certain algebras.

\section{Preliminaries}

Throughout this paper, $\mathcal{A}$ is an abelian category and all subcategories of $\mathcal{A}$ involved
are full, additive and closed under isomorphisms and direct summands, and all functors between categories are additive.
For a subclass $\mathcal{U}$ of $\mathcal{A}$, we use $\add \mathcal{U}$ to
denote the subcategory of $\mathcal{A}$ consisting of direct summands of finite direct sums of objects
in $\mathcal{U}$.

\subsection{The extension dimension of an abelian category}

Let $\mathcal{U}_1,\mathcal{U}_2,\cdots,\mathcal{U}_n$ be subcategories of $\mathcal{A}$. Define
$$\mathcal{U}_1\bullet \mathcal{U}_2:={\add}\{A\in \mathcal{A}\mid {\rm there \;exists \;an\;exact\; sequence \;}
0\rightarrow U_1\rightarrow  A \rightarrow U_2\rightarrow 0$$
$$\ \ \ \ \ \ \ \ \ \ \ \  \ \ \ \ \ \ \ \ \ \ \ \ \ \ \ \ \
\ \ \ \ \ \ \ \ \ \ \ \  \ \ \ \ \ \ \ \ \ \ \ \ \ \ \ \ \
{\rm in}\ \mathcal{A}\ {\rm with}\; U_1 \in \mathcal{U}_1 \;{\rm and}\;
U_2 \in \mathcal{U}_2\}.$$
For any subcategories $\mathcal{U},\mathcal{V}$ and $\mathcal{W}$ of $\mathcal{A}$, by \cite[Proposition 2.2]{dao2014radius} we have
$$(\mathcal{U}\bullet\mathcal{V})\bullet\mathcal{W}=\mathcal{U}\bullet(\mathcal{V}\bullet\mathcal{W}).$$
Inductively, we define
\begin{align*}
\mathcal{U}_{1}\bullet  \mathcal{U}_{2}\bullet \dots \bullet\mathcal{U}_{n}:=
\add \{A\in \mathcal{A}\mid {\rm there \;exists \;an\;exact\; sequence}\
0\rightarrow U\rightarrow  A \rightarrow V\rightarrow 0  \\
{\rm in}\ \mathcal{A}\ {\rm with}\; U \in \mathcal{U}_{1} \;{\rm and}\;
V \in  \mathcal{U}_{2}\bullet \dots \bullet\mathcal{U}_{n}\}.
\end{align*}
For a subcategory $\mathcal{U}$ of $\mathcal{A}$, set
$[\mathcal{U}]_{0}=0$, $[\mathcal{U}]_{1}=\add\mathcal{U}$,
$[\mathcal{U}]_{n}=[\mathcal{U}]_1\bullet [\mathcal{U}]_{n-1}$ for any $n\geqslant 2$,
and $[\mathcal{U}]_{\infty}=\mathop{\bigcup}_{n\geqslant 0}[\mathcal{U}]_{n}$ (\cite{beligiannis2008some}).

\begin{definition}\label{def-2.1}
{\rm (\cite{beligiannis2008some})
The {\bf extension dimension} $\extdim \mathcal{A}$ of $\mathcal{A}$ is defined to be
$$\extdim{\mathcal{A}}:=\inf\{n\geqslant 0\mid\mathcal{A}=[A]_{n+1}\ {\rm with}\ A\in\mathcal{A}\},$$
or $\infty$ if no such an integer exists.}
\end{definition}


The following lemma is used frequently in the sequel.

\begin{lemma}\label{lem-2.2} {\rm (\cite[Corollary 2.3(1)]{zheng2019extension})}
For any $A_{1},A_{2}\in \mathcal{A}$ and $m,n\geqslant 1$, we have
$$[A_{1}]_{m}\bullet [A_{2}]_{n}\subseteq [A_{1}\oplus A_{2}]_{m+n}.$$
\end{lemma}

\subsection{The dimension of a triangulated category}

Let $\T$ be a triangulated category and $\I \subseteq {\rm Ob}\T$.
Let $\langle \I \rangle_{1}$ be the full subcategory of $\T$ consisting
of all direct summands of finite direct sums of shifts of objects in $\I$.
Given two subclasses $\I_{1}, \I_{2}\subseteq {\rm Ob}\T$, we use $\I_{1}*\I_{2}$
to denote the full subcategory of all extensions between them, that is,
$$\I_{1}*\I_{2}=\{ X\mid\text{there exists a distinguished triangle}\
X_{1} \longrightarrow X \longrightarrow X_{2}\longrightarrow X_{1}[1]$$
$$\ \ \ \ \ \ \ \ \ \ \ \ \ \ \ \ \ \ \ \ \ \ \ \ \ \ \ \ \ \ \ \ \ \
\ \ \ \ \ \ \ \ \ \ \ \ \ \ \ \ \ \ \ \ \ \ \ \
{\rm in}\; \T\; {\rm with}\; X_{1}\in \I_{1}\; {\rm and}\; X_{2}\in \I_{2}\}.$$
We write $\I_{1}\diamond\I_{2}:=\langle\I_{1}*\I_{2} \rangle_{1}.$
Then for any subclasses $\I_{1}, \I_{2}$ and $\I_{3}$ of $\T$, we have
$$(\I_{1}\diamond\I_{2})\diamond\I_{3}=\I_{1}\diamond(\I_{2}\diamond\I_{3})$$
by the octahedral axiom. In addition, we write
\begin{align*}
\langle \I \rangle_{0}:=0,\;
\langle \I \rangle_{1}:=\langle \I \rangle\; {\rm and}\;
\langle \I \rangle_{n+1}:=\langle \I \rangle_{n}\diamond\langle \I \rangle_{1}\;{\rm for\; any \;}n\geqslant 1.
\end{align*}

\begin{definition}\label{def-2.3}
{\rm (\cite{oppermann2007lower1, oppermann2009lower, rouquier2006representation})
\begin{itemize}
\item[$(1)$] The {\bf dimension} $\dim \T$ of a triangulated category $\T$ is defined to be
$$\dim{\mathcal{T}}:=\inf\{n\geqslant 0\mid\mathcal{T}=\langle T\rangle_{n+1}\ \text{\rm for some}\ T\in\mathcal{T}\},$$
or $\infty$ if no such an integer exists.
\item[$(2)$] For a subcategory $\C$ of $\T$, the {\bf dimension} of $\C$ is defined to be
$$\dim_{\T}\C:=\inf\{n\geqslant 0\mid \C \subseteq\langle T\rangle_{n+1}\ \text{\rm for some}\ T\in\mathcal{T}\},$$
or $\infty$ if no such an integer exists.
\item[$(3)$] For an abelian category $\mathcal{A}$, the bounded derived category $D^b(\mathcal{A})$ of
$\mathcal{A}$ is a triangulated category. We call $\derdim \mathcal{A}:=\dim D^b(\mathcal{A})$
the {\bf derived dimension} of $\mathcal{A}$.
\end{itemize}}
\end{definition}

The following lemma is an analogue of Lemma \ref{lem-2.2}.

\begin{lemma}\label{lem-2.4} {\rm (\cite[Lemma 7.3]{psaroudakis2014homological})}
For any $T_{1},T_{2}\in \mathcal{T}$ and $m,n\geqslant 1$, we have
$$\langle T_{1} \rangle _{m}\diamond \langle T_{2} \rangle _{n} \subseteq \langle T_{1}\oplus T_{2} \rangle _{m+n}.$$
\end{lemma}

\subsection{ Radical layer lengths and torsion pairs}

We recall some notions from \cite{huard2009finitistic}.
Let $\mathcal{C}$ be a {\bf length-category}, that is, $\mathcal{C}$
is an abelian, skeletally small category and every object of $\mathcal{C}$ has a finite composition series.
We use $\End_{\mathbb{Z}}(\mathcal{C})$ to denote the category of all additive functors from
$\mathcal{C}$ to $\mathcal{C}$, and use $\rad$ to denote the Jacobson radical lying in $\End_{\mathbb{Z}}(\mathcal{C})$.
For any $\alpha\in\End_{\mathbb{Z}}(\mathcal{C})$, set the {\bf $\alpha$-radical functor} $F_{\alpha}:=\rad\circ \alpha$.

\begin{definition}\label{def-2.5} {\rm (\cite[Definition 3.1]{huard2013layer})
For any $\alpha, \beta \in \End_{\mathbb{Z}}(\mathcal{C})$, we define
the {\bf $(\alpha,\beta)$-layer length} $\ell\ell_{\alpha}^{\beta}:\mathcal{C} \longrightarrow \mathbb{N}\cup \{\infty\}$ via
$\ell\ell_{\alpha}^{\beta}(M)=\inf\{ i \geqslant 0\mid \alpha \circ \beta^{i}(M)=0 \}$; and
the {\bf $\alpha$-radical layer length} $\ell\ell^{\alpha}:=\ell\ell_{\alpha}^{F_{\alpha}}$.}
\end{definition}

\begin{lemma}\label{lem-2.6} {\rm (\cite[Lemma 2.6]{zheng2017upper})}
Let $\alpha,\beta \in\End_{\mathbb{Z}}(\mathcal{C}) $.
For any $M\in \mathcal{C}$, if $\ell\ell_{\alpha}^{\beta}(M)=n$, then $\ell\ell_{\alpha}^{\beta}(M)=\ell\ell_{\alpha}^{\beta}(\beta^{j}(M))+j$
for any $0 \leqslant j\leqslant n$; in particular, if $\ell\ell^{\alpha}(M)=n$, then $\ell\ell^{\alpha}(F_{\alpha}^{n}(M))=0$.
\end{lemma}

Recall that a {\bf torsion pair} (or {\bf torsion theory}) for $\mathcal{C}$
is a pair of classes $(\mathcal{T},\mathcal{F})$ of objects in $\mathcal{C}$ satisfying the following conditions.
\begin{enumerate}
\item[(1)] $\Hom_{\mathcal{C}}(M,N)=0$ for any $M\in\mathcal{T}$ and $N\in\mathcal{F}$;
\item[(2)] an object $X \in \mathcal{C}$ is in $\mathcal{T}$ if $\Hom_{\mathcal{C}}(X,-)|_{\mathcal{F}}=0$;
\item[(3)] an object $Y\in\mathcal{C}$ is in $\mathcal{F}$ if $\Hom_{\mathcal{C}}(-,Y)|_{\mathcal{T}}=0$.
\end{enumerate}

For a subfunctor $\alpha$ of the identity functor $1_{\C}$, we write $q_{\alpha}:=1_{\mathcal{C}}/\alpha$.
Let $(\mathcal{T},\mathcal{F})$ be a torsion pair for $\mathcal{C}$.
Recall that the {\bf torsion radical} $t$ is a functor in $\End_{\mathbb{Z}}(\mathcal{C})$ such that
$$0 \longrightarrow  t(M)\longrightarrow M \longrightarrow  q_{t}(M)\longrightarrow 0$$
is a short exact sequence and $q_{t}(M)=M/t(M)\in \mathcal{F}$.

\subsection{Homologically finite subcategories}

Let $\Lambda$ be an artin algebra and $\mod\Lambda$ the category of finitely generated right $\Lambda$-modules.
Let $M, N\in\mod \Lambda$. Recall that a homomorphism $f: N \to M$ in $\mod \Lambda$ is called {\bf right minimal}
if every $h \in \End(N_{\Lambda})$ such that $fh=f$ is an automorphism. Let $\X$ be a subcategory of $\mod\Lambda$ and
$M \in \mod\Lambda$. A homomorphism $f: X\to M$ in $\mod \Lambda$ is called a {\bf right $\X$-approximation}
of $M$ if $X\in \mathcal{X}$ and the sequence $\Hom_{\Lambda}(X',f)$ is epic for any $X'\in\mathcal{X}$.
The category $\X$ is called a {\bf contravariantly finite subcategory} of $\mod \Lambda$ if each module in $\mod \Lambda$
admits a right $\mathcal{X}$-approximation. Dually, ({\bf minimal}) {\bf left $\X$-approximations} and {\bf covariantly finite
subcategories} are defined (\cite{AuslanderApplications}).
If $f: X\to M$ in $\mod \Lambda$ is a minimal right $\mathcal{X}$-approximation of $M$ and $n\geqslant 1$, then we
write $\Omega^{1}_{\mathcal{X}}(M):=\Ker f$ and $\Omega^{n}_{\X}(M):=\Omega_{\X}^1(\Omega^{n-1}_{\X}(M))$.
Dually, if $f: M\to X$ in $\mod \Lambda$ is a minimal left $\mathcal{X}$-approximation of $M$, then we
write $\Omega^{-1}_{\mathcal{X}}(M):=\Coker f$ and $\Omega^{-n}_{\X}(M):=\Omega_{\X}^{-1}(\Omega^{-(n-1)}_{\X}(M))$.
In particular, $\Omega^{0}_{\X}(M):=M$.

Recall that a subcategory $\X$ of $\mod \Lambda$ is called {\bf resolving} if $\X$ contains all projective modules
in $\mod\Lambda$, and $\X$ is closed under extensions and kernels of epimorphisms; and $\X$
is called {\bf coresolving} if $\X$ contains all injective modules
in $\mod\Lambda$, and $\X$ is closed under extensions and cokernels of monomorphisms.

Let $M\in\mod \Lambda$. If $\X$ is a contravariantly finite and resolving subcategory of $\mod \Lambda$,
then there exists an exact sequence
$$\cdots\rightarrow X_{n}\xrightarrow{f_{n}} X_{n-1}\rightarrow\cdots
\rightarrow X_{1}\xrightarrow{f_{1}}X_{0}\xrightarrow{f_{0}} M\rightarrow 0$$
in $\mod \Lambda$ such that each $X_{i}\rightarrow \Im f_{i}$ is a (minimal) right $\X$-approximation of $\Im f_i$.
In this case, we call this exact sequence a ({\bf minimal}) {\bf $\X$-resolution} of $M$.

Let $\X$ be a subcategory of $\mod\Lambda$ and $M\in\mod\Lambda$.
If $\X$ is contravariantly finite and resolving, then the {\bf $\X$-projective dimension} $\pd_{\X}M$
of $M$ is defined as $\inf\{n\mid \Omega^{n}_{\X}(M)\in\X\}$,
and set $\pd_{\X}M=\infty$ if no such an integer exists. Dually,
if $\X$ is covariantly finite and coresolving, then the
{\bf $\X$-injective dimension} $\id_{\X}M$ of $M$ is defined as $\inf\{n\mid \Omega^{-n}_{\X}(M)\in\X\}$,
and set $\id_{\X}M=\infty$ if no such an integer exists.
In particular, set $\pd_{\X}M=-1=\id_{\X}M$ if $M=0$.

\section{Main results}

\subsection{A relation between the derived and extension dimensions}

The following lemma is essentially contained in the proof of \cite[Theorem]{han2009derived}.

\begin{lemma}\label{lem-3.1}
For any bounded complex $X=(X_n,f_n)_{n\in\mathbb{Z}}$ over $\mathcal{A}$, we have
$$X\in \langle\oplus_{n\in \mathbb{Z}} Y_{n}[n] \rangle _{1}\diamond \langle\oplus_{n\in \mathbb{Z}}Z_{n}[n] \rangle _{1}$$
in $D^{b}(\mathcal{A})$, where $Y_{n}=\Ker f_n$ and $Z_{n}=\Im f_n$ for any $n\in\mathbb{Z}$,
and both $\oplus_{n\in \mathbb{Z}} Y_{n}[n]$ and $\oplus_{n\in \mathbb{Z}} Z_{n}[n]$ have only finitely many nonzero summands.
\end{lemma}

We also need the following lemma.

\begin{lemma}\label{lem-3.2}
\begin{itemize}
\item[]
\item[$(1)$] For an object $M\in\mathcal{A}$, if $M\in [T]_{n+1}$ for some $T\in \mathcal{A}$, then
$M\in \langle T \rangle _{n+1}$ in $D^{b}(\mathcal{A})$ with $M$ and $T$ stalk complexes in degree zero.
\item[$(2)$] $\dim_{D^b(\mathcal{A})}\mathcal{A}\leqslant\min\{\derdim \mathcal{A}, \extdim \mathcal{A}\}$.
\end{itemize}
\end{lemma}

\begin{proof}
(1) Let $M\in [T]_{n+1}$. Then we have the following exact sequence
$$0\longrightarrow  Y_{i}\longrightarrow Z_{i-1}\oplus Z_{i-1}'\longrightarrow Z_{i}\longrightarrow 0$$
in $\mathcal{A}$ with $Z_{0}=M$, $Y_{i}\in [T]_{1} $ and $Z_{i}\in [T]_{n+1-i}$ for any $1\leqslant i \leqslant n$.
It induces a triangle
$$Y_{i}\longrightarrow Z_{i-1}\oplus Z_{i-1}'\longrightarrow Z_{i}\longrightarrow Y_{i}[1]$$
in $D^{b}(\mathcal{A})$ for any $1\leqslant i \leqslant n$.
Thus $\langle Z_{i-1} \rangle _{1}\subseteq  \langle Y_{i} \rangle _{1}\diamond \langle Z_{i} \rangle _{1}$
for any $1 \leqslant i \leqslant n$, and therefore
\begin{align*}
M\in\langle Z_{0} \rangle _{1}\subseteq& \langle Y_{1} \rangle _{1}\diamond\langle Y_{2} \rangle _{1}\diamond
\cdots \diamond\langle Y_{n} \rangle _{1}\diamond\langle Z_{n} \rangle _{1}\\
\subseteq& \underbrace{\langle T \rangle _{1}\diamond\langle T \rangle _{1}\diamond
\cdots \diamond\langle T \rangle _{1}\diamond\langle T \rangle _{1}}_{n}\\
\subseteq&\langle T \rangle _{n+1}.  \;\;\;\;\;\;\;\;\;\;\;\;\;\;\;\;(\text{by Lemma }\ref{lem-2.4})
\end{align*}

(2) By (1) and Definition \ref{def-2.3}.
\end{proof}

We establish a relation between the derived and extension dimensions of $\mathcal{A}$.

\begin{theorem}\label{thm-3.3}
We have
$$\dim_{D^b(\mathcal{A})}\mathcal{A}\leqslant\derdim \mathcal{A}\leqslant2\dim_{D^b(\mathcal{A})}\mathcal{A}+1\leqslant 2\extdim \mathcal{A}+1.$$
\end{theorem}

\begin{proof}
By Lemma \ref{lem-3.2}(2), it suffices to prove $\derdim \mathcal{A}\leqslant2\dim_{D^b(\mathcal{A})}\mathcal{A}+1$.
Without loss of generality, suppose that $\dim_{D^b(\mathcal{A})}\mathcal{A}=m<\infty$ and
$\A\subseteq\langle T\rangle_{m+1}$ for some $T\in D^b(\mathcal{A})$.
For any $X\in D^{b}(\A)$, by Lemma \ref{lem-3.1} we have
$$X\in \langle\oplus_{n\in \mathbb{Z}} Y_{n}[n] \rangle _{1}\diamond \langle\oplus_{n\in \mathbb{Z}}Z_{n}[n] \rangle _{1}$$
in $D^{b}(\mathcal{A})$, where $Y_{n},Z_{n}\in\mathcal{A}$ for any $n\in\mathbb{Z}$,
and $\oplus_{n\in \mathbb{Z}} Y_{n}[n]$ and $\oplus_{n\in \mathbb{Z}} Z_{n}[n]$ have only finitely many nonzero direct summands.

By Lemma \ref{lem-3.2}, we have $Z_{n}[n]\in \langle T \rangle_{m+1}$
and $\oplus_{n\in \mathbb{Z}} Z_{n}[n]\in\langle T \rangle_{m+1}$, and then
$\langle \oplus_{n\in \mathbb{Z}} Z_{n}[n]\rangle_{1} \subseteq\langle T \rangle_{m+1}$.
Similarly, we have $\langle \oplus_{n\in \mathbb{Z}} Y_{n}[n]\rangle_{1} \subseteq \langle T \rangle_{m+1}$.
It follows from Lemma \ref{lem-2.4} that
$$X\in \langle\oplus_{n\in \mathbb{Z}} Y_{n}[n] \rangle _{1}\diamond \langle\oplus_{n\in
\mathbb{Z}}Z_{n}[n] \rangle _{1}\subseteq  \langle T \rangle_{m+1}\diamond   \subseteq \langle T \rangle_{m+1}
\subseteq \langle T \rangle_{2m+2}$$
and $\derdim \mathcal{A} \leqslant 2m+1$.
\end{proof}

Let $\Lambda$ be an artin algebra. For simplicity, we write
$$\extdim\Lambda:=\extdim\mod\Lambda\ {\rm and}\ \derdim\Lambda:=\derdim\mod\Lambda.$$
Recall from \cite{A} that the {\bf representation dimension} $\repdim\Lambda$ of $\Lambda$ is defined as
\begin{equation*}
\repdim\Lambda:=
\begin{cases}
\inf\{\gldim\End_{\Lambda}(M)\mid M\ \text{is a generator-cogenerator for}\ \mod\Lambda\},\ \text{if}\ \Lambda\ \text{is non-semisimple;}\\
1,\;\text{if}\ \Lambda\ \text{is semisimple.}
\end{cases}
\end{equation*}
In \cite[p.70]{oppermann2007lower1}, Oppermann posed an open question:
\newline\centerline{\it Are there non-semisimple artin algebras $\Lambda$, such that the equality holds in the inequality}
$$\repdim \Lambda \geqslant\derdim\Lambda?$$
Let $\Lambda$ be a non-semisimple artin algebra of finite representation type. It is well known that $\repdim \Lambda=2$.
By \cite[Theorem]{han2009derived}, we have $\derdim\Lambda\leqslant 1$. Thus, in this case, $\repdim \Lambda>\derdim\Lambda$.
Here, we give the following example in which $\Lambda$ is of infinite representation type such that $\repdim \Lambda>\derdim\Lambda$.

\begin{example}\label{exa-3.4}
{\rm Let $\Lambda$ be the Beilinson algebra $kQ/I$ with $Q$ the quiver
$$\xymatrix{
&0 \ar@/_1pc/[r]_{x_{n}}\ar@/^1pc/[r]^{x_{0}}_{\vdots}
&1\ar@/_1pc/[r]_{x_{n}}\ar@/^1pc/[r]^{x_{0}}_{\vdots}
&2\ar@/_1pc/[r]_{x_{n}}\ar@/^1pc/[r]^{x_{0}}_{\vdots}
&3&\cdots &n-1\ar@/_1pc/[r]_{x_{n}}\ar@/^1pc/[r]^{x_{0}}_{\vdots}&n
}$$
and $I=(x_{i}x_{j}-x_{j}x_{i})$ (where $0 \leqslant i, j \leqslant n$) (see \cite[Example 3.7]{oppermann2010representation}).
Then $\gldim \Lambda=n$.
From \cite[Theorem 4.15]{oppermann2010representation} and its proof, we know that
$\repdim \Lambda =n+2$ and $\dim_{D^b(\mod\Lambda)}\mod\Lambda\geqslant n$.
Then $\extdim\Lambda \geqslant n$ by Lemma \ref{lem-3.2}(2). On the other hand, we have
$\extdim\Lambda\leqslant\repdim \Lambda-2=(n+2)-2=n$ by \cite[Corollary 3.6]{zheng2019extension}.
Thus $\extdim \Lambda=n$. Note that $\derdim\Lambda\leqslant\gldim\Lambda=n$
(see \cite[Lemma 2.11]{oppermann2009lower} or \cite[Proposition 7.4]{rouquier2008dimensions}).
Now Lemma \ref{lem-3.2}(2) induces the following equality
$$\dim_{D^{b}(\mod \Lambda)}\mod\Lambda=\extdim\Lambda=\derdim\Lambda=\gldim \Lambda=\repdim \Lambda-2 =n.$$}
\end{example}

\subsection{Syzygies and cosyzygies}

Let $M\in\mathcal{A}$. If $\mathcal{A}$ has enough projective objects, then there exists an exact sequence
$$\cdots\to P_i\to\cdots\to P_1\to P_0\to M\to 0$$
in $\mathcal{A}$ with all $P_i$ projective. We write $\Omega^n(M):=\Im(P_n\to P_{n-1})$ for any $n\geqslant 1$.
Dually, if $\mathcal{A}$ has enough injective objects, then there exists an exact sequence
$$0\to M \to I^0\to I^1 \to\cdots\to I^i\to\cdots$$
in $\mathcal{A}$ with all $I^i$ injective. We write $\Omega^{-n}(M):=\Im(I^{n-1}\to I^n)$ for any $n\geqslant 1$.
In particular, we write $\Omega^{0}(M):=M$.

%
%

\begin{lemma}\label{lem-3.5}
\begin{itemize}
\item[]
\item[$(1)$] If $\mathcal{A}$ has enough projective objects and
$$0\longrightarrow M\longrightarrow X^{0} \longrightarrow  X^{1}\longrightarrow\cdots\longrightarrow
X^{n} \longrightarrow 0,$$
is an exact sequence in $\mathcal{A}$ with $n\geqslant 0$, then
$$M\in[\Omega^{n}(X^{n})]_{1}\bullet[\Omega^{n-1}(X^{n-1})]_{1}\bullet\cdots\bullet
[\Omega^{1}(X^{1})]_{1}\bullet[X^{0}]_{1}\subseteq[\oplus_{i=0}^{n}\Omega^{i}(X^{i})]_{n+1}.$$
\item[$(2)$] If $\mathcal{A}$ has enough injective objects and
$$0\longrightarrow X_{n}\longrightarrow
\cdots \longrightarrow X_{1} \longrightarrow  X_{0} \longrightarrow M \longrightarrow 0,$$
is an exact sequence in $\mathcal{A}$ with $n\geqslant 0$, then
$$M\in[X_{0}]_{1}\bullet[\Omega^{-1}(X_{1})]_{1}\bullet\cdots\bullet[\Omega^{-n}(X_{n})]_{1}
\subseteq[\oplus_{i=0}^{n}\Omega^{-i}(X_{i})]_{n+1}.$$
\end{itemize}
\end{lemma}

\begin{proof}
The assertion (1) is \cite[Lemma 5.8]{dao2014radius}, and (2) is dual to (1).
\end{proof}

\begin{lemma}\label{lem-3.6}
Let $X, Y\in \mathcal{A}$ satisfy $[X]_{n_{1}}\subseteq [Y]_{n_{2}}$ with $n_{1}, n_{2}\geqslant 1$.
Then for any $m\geqslant 0$, we have
\begin{itemize}
\item[$(1)$] If $\mathcal{A}$ has enough projective objects, then $[\Omega^{m}(X)]_{n_{1}}\subseteq[\Omega^{m}(Y)]_{n_{1}n_{2}}$.
\item[$(2)$] If $\mathcal{A}$ has enough injective objects, then $[\Omega^{-m}(X)]_{n_{1}}\subseteq[\Omega^{-m}(Y)]_{n_{1}n_{2}}$.
\end{itemize}

\end{lemma}
\begin{proof}
We only prove (1), and we get (2) dually.

We proceed by induction on $n_{1}$. Let $n_1=1$ and $W\in [\Omega^{m}(X)]_{1}$. Then
$$W\oplus W_{1}\cong (\Omega^{m}(X))^{(l)}(\cong \Omega^{m}(X^{(l)})$$
for some $l\geqslant 1$ and $W_{1}\in \A$.
Since $X^{(l)}\in [X]_{1} \subseteq [Y]_{n_{2}}$ by assumption,
we have the following exact sequences
$$0\longrightarrow Y^{'}_{1}\longrightarrow X^{(l)}\oplus Z_1 \longrightarrow Y_{1}\longrightarrow 0,$$
$$0\longrightarrow Y^{'}_{2}\longrightarrow  Y_{1}\oplus Z_2 \longrightarrow Y_{2}\longrightarrow 0,$$
$$0\longrightarrow Y^{'}_{3}\longrightarrow Y_{2}\oplus Z_3 \longrightarrow Y_{3}\longrightarrow 0,$$
$$\cdots\cdots\cdots$$
$$0\longrightarrow Y^{'}_{n_{2}-1}\longrightarrow Y_{n_{2}-2}\oplus Z_{n_{2}-1} \longrightarrow Y_{n_{2}-1}\longrightarrow 0,$$
where $Z_i\in\mathcal{A}$, $Y^{'}_i\in[Y]_{1}$ and $Y_i\in[Y]_{n_{2}-i}$ for any $1\leqslant i\leqslant n_2-1$. By the horseshoe lemma, we have
$$0\longrightarrow\Omega^{m}(Y^{'}_{1})\longrightarrow\Omega^{m}(X^{(l)})\oplus\Omega^{m}(Z_1)\oplus P_{1}\longrightarrow\Omega^{m}(Y_{1})\longrightarrow 0,$$
$$0\longrightarrow\Omega^{m}(Y^{'}_{2})\longrightarrow\Omega^{m}(Y_{1})\oplus\Omega^{m}(Z_2)\oplus P_{2}\longrightarrow\Omega^{m}(Y_{2})\longrightarrow 0,$$
$$0\longrightarrow\Omega^{m}(Y^{'}_{3})\longrightarrow\Omega^{m}(Y_{2})\oplus\Omega^{m}(Z_3)\oplus P_{3}\longrightarrow\Omega^{m}(Y_{3})\longrightarrow 0,$$
$$\cdots\cdots\cdots$$
$$0\longrightarrow \Omega^{m}(Y^{'}_{n_{2}-1})\longrightarrow \Omega^{m}(Y_{n_{2}-2})\oplus\Omega^{m}(Z_{n_{2}-1})\oplus P_{n_{2}-1}
\longrightarrow \Omega^{m}(Y_{n_{2}-1})\longrightarrow 0,$$
where all $P_{i}$ are projective. Then we have
\begin{align*}
\Omega^{m}(X^{(l)}) \in &[\Omega^{m}(Y^{'}_{1})]_{1}\bullet[\Omega^{m}(Y_{1})]_{1}\\
\subseteq &[\Omega^{m}(Y^{'}_{1})]_{1}\bullet[\Omega^{m}(Y^{'}_{2})]_{1}\bullet[\Omega^{m}(Y_{2})]_{1}\\
\subseteq &[\Omega^{m}(Y^{'}_{1})]_{1}\bullet[\Omega^{m}(Y^{'}_{2})]_{1}\bullet [\Omega^{m}(Y^{'}_{3})]_{1}\bullet[\Omega^{m}(Y_{3})]_{1}\\
&\cdots\cdots\cdots\\
\subseteq &[\Omega^{m}(Y^{'}_{1})]_{1}\bullet[\Omega^{m}(Y^{'}_{2})]_{1}\bullet\cdots\bullet
[\Omega^{m}(Y^{'}_{n_{2}-1})]_{1}\bullet[\Omega^{m} (Y_{n_{2}-1})]_{1}\\
\subseteq &[(\oplus_{i=1}^{n_{2}-1} \Omega^{m}(Y^{'}_{i}))\oplus \Omega^{m}(Y_{n_{2}-1})]_{n_{2}}\\
\subseteq &[\Omega^{m}(Y)]_{n_{2}},
\end{align*}
and hence $W\in [\Omega^{m}(Y)]_{n_{2}}$. The case for $n_1=1$ is proved.

Now suppose $n_1\geqslant 2$ and $W\in [\Omega^{m}(X)]_{n_1}$.
By \cite[Proposition 2.2(3)]{zheng2019extension} and assumption, we have
$$[X]_{1}\subseteq[X]_{n_1-1}\subseteq[X]_{n_1}\subseteq [Y]_{n_{2}}.$$
Then by the induction hypothesis, we have
$$[\Omega^{m}(X)]_{1}\subseteq[\Omega^{m}(Y)]_{n_{2}}\ \text{and}\ [\Omega^{m}(X)]_{n_1-1}\subseteq[\Omega^{m}(Y)]_{(n_1-1)n_{2}}.$$
Thus
\begin{align*}
W \in &[ \Omega^{m}(X)]_{n_1}\\
=&[ \Omega^{m}(X)]_{1} \bullet [ \Omega^{m}(X)]_{n_1-1}\\
\subseteq &[\Omega^{m}(Y)]_{n_{2}}\bullet [\Omega^{m}(Y)]_{(n_1-1)n_{2}} \ {\rm  (by \;\cite[Proposition\; 2.2(1)]{zheng2019extension})}\\
= & [\Omega^{m}(Y)]_{n_1n_{2}.}
\end{align*}
The proof is finished.
\end{proof}

\subsection{$t_{\V}$-radical layer length and extension dimension}

From now on, $\La$ is an artin algebra.
Then the category $\mod\Lambda$ of finitely generated right $\Lambda$-modules
is a length-category. We use $\rad \La$ to denote the Jacobson radical of $\La$.
For a module $M$ in $\mod\La$, we use $\top M$ to denote the top of $M$.

Let $\mathcal{S}$ be the set of all pairwise non-isomorphic simple modules in $\mod\Lambda$ and $\mathcal{V}$ a subset of $\mathcal{S}$.
%
We write $\mathfrak{F}\,(\mathcal{V}):=\{M\in\mod\Lambda\mid$ there exists a finite chain
$$0=M_0\subseteq M_1\subseteq \cdots\subseteq M_m=M$$ of submodules of $M$
such that each quotient $M_i / M_{i-1}$ is isomorphic to some module in $\mathcal{V}\}$.
By \cite[Lamma 5.7 and Proposition 5.9]{huard2013layer}, we have that
$(\T_{\mathcal{V}}, \mathfrak{F}(\mathcal{V}))$ is a torsion pair, where
$$\T_{\mathcal{V}}=\{M \in \mod \Lambda\mid\top M\in \add \mathcal{V}'\ {\rm with}\ \mathcal{V}'=\mathcal{S}\backslash\mathcal{V}\}.$$
We use $t_\mathcal{V}$ to denote the torsion radical of the torsion pair $(\mathcal{T}_\mathcal{V}, \mathfrak{F}(\mathcal{V}))$.
Then $t_{\mathcal{V}}(M)\in \T_{\mathcal{V}}$ and $q_{_{t_{\mathcal{V}}}}(M)\in\mathfrak{F}(\mathcal{V})$ for any
$M\in \mod \Lambda$. By \cite[Proposition 5.3]{huard2013layer}, we have
$$\mathfrak{F}(\mathcal{V})=\{ M\in \mod \Lambda \mid t_{\mathcal{V}}(M)=0\},$$
$$\T_{\mathcal{V}}=\{ M\in \mod \La\mid t_{\mathcal{V}}(M)\cong M\}.$$





We have the following easy observation.

\begin{lemma}\label{lem-3.7}
Let $\mathcal{V}$ be a subset of $\mathcal{S}$. Then for any $M\in \mod \Lambda$ and $i\geqslant 0$,
we have the following exact sequences
$$0 \rightarrow t_{\mathcal{V}}F^{i}_{t_{\mathcal{V}}}(M)\rightarrow
F^{i}_{t_{\mathcal{V}}}(M)\rightarrow q_{t_{\mathcal{V}}}F^{i}_{t_{\mathcal{V}}}(M) \rightarrow 0,$$
$$0 \rightarrow F^{i+1}_{t_{\mathcal{V}}}(M)\rightarrow
t_{\mathcal{V}}F^{i}_{t_{\mathcal{V}}}(M)\rightarrow \top t_{\mathcal{V}}F^{i}_{t_{\mathcal{V}}}(M) \rightarrow 0,$$
where $F_{t_{\mathcal{V}}}=\rad\circ t_{\mathcal{V}}$.
\end{lemma}

\begin{lemma}\label{lem-3.8}
Let $\mathcal{V}$ be a subset of $\mathcal{S}$ and $M\in \mod \Lambda$.
\begin{itemize}
\item[$(1)$] If $\ell\ell^{t_{\mathcal{V}}}(M)=0$, then $M\in\mathfrak{F}(\mathcal{V})$
and $M\cong q_{t_\mathcal{V}}(M)$.
\item[$(2)$] If $\ell\ell^{t_{\mathcal{V}}}(\Lambda)=n$, then $\ell\ell^{t_{\mathcal{V}}}(F^{n}_{t_{\mathcal{V}}}(M))=0$;
in particular, $F^{n}_{t_{\mathcal{V}}}(M)\in\mathfrak{F}(\mathcal{V})$.
\item[$(3)$] If $M=\oplus^{m}_{i=1}M_{i}$, then
$\ell\ell^{t_{\mathcal{V}}}(M)=\max\{\ell\ell^{t_{\mathcal{V}}}(M_{i})\;|\;1\leqslant i \leqslant n\}$.
\end{itemize}
\end{lemma}

\begin{proof}
(1) If $\ell\ell^{t_{\mathcal{V}}}(M)=0$, then $t_{\V}(M)=0$ and $M\in\mathfrak{F}(\mathcal{V})$.
Putting $i=0$ in the first exact sequence in Lemma \ref{lem-3.7}, we have $M\in\mathfrak{F}(\mathcal{V})$.

(2) By \cite[Lemma 3.4(b)]{huard2013layer}, we have $\ell\ell^{t_{\mathcal{V}}}(M)\leqslant\ell\ell^{t_{\mathcal{V}}}(\Lambda)=n$.
Thus $\ell\ell^{t_{\mathcal{V}}}(F^{n}_{t_{\mathcal{V}}}(M))=0$ by Lemma \ref{lem-2.6}.

(3) It follows from \cite[Lemma 3.4(a)]{huard2013layer}.
\end{proof}

\begin{lemma}\label{lem-3.9}
Let $\mathcal{V}$ be a subset of $\mathcal{S}$. Then the following statements are equivalent.
\begin{itemize}
\item[$(1)$] $\mathcal{V}=\mathcal{S}$.
\item[$(2)$] $\ell\ell^{t_{\mathcal{V}}}(\Lambda)=0$.
\item[$(3)$] $\ell\ell^{t_{\mathcal{V}}}(M)=0$ for any $M\in\mod\Lambda$.
\item[$(4)$] $\mathfrak{F}(\mathcal{V})=\mod\Lambda$.
\end{itemize}
\end{lemma}

\begin{proof}
The implications $(1)\Leftrightarrow (4)$ and $(3)\Rightarrow (2)$ are trivial.
By \cite[Proposition 3.4]{zheng2017upper}, we have $(2)\Rightarrow (3)$.

Since $\mathfrak{F}(\mathcal{V})=\mod\Lambda$ if and only if $t_\mathcal{V}(M)=0$ for any $M\in\mod\Lambda$, we have
$(3)\Leftrightarrow (4)$.
\end{proof}

For a subcategory $\mathcal{X}$ of $\mod\Lambda$,
we write $$\X^{\perp}:=\{ Z\in \mod \Lambda\mid\Ext^{i}_{\Lambda}(X, Z)=0
\text{\;for any }X\in \X \;\text{and}\;i\geqslant 1\}.$$

\begin{lemma}\label{lem-3.10}
Let $\X$ be a contravariantly finite and resolving subcategory of $\mod\Lambda$ and
$$0 \longrightarrow C_{1} \longrightarrow C_{2}\longrightarrow C_{3}\longrightarrow 0$$
an exact sequence in $\mod \Lambda$. Then
there exists the following commutative diagram with exact columns and rows
\[\xymatrix{
&0 \ar[d]&0\ar[d]&&0\ar[d]&0\ar[d]&0\ar[d]&\\
\cdots \ar[r]& X^n_1\ar[r]\ar[d]& X^{n-1}_1\ar[r]\ar[d]&\cdots\ar[r]
& X^1_1\ar[r]\ar[d]& X^0_1\ar[r]\ar[d]&C_{1}\ar[r]\ar[d]&0\\
\cdots \ar[r]& X^n_2\ar[r]\ar[d]& X^{n-1}_2\ar[r]\ar[d]&\cdots\ar[r]
& X^1_2\ar[r]\ar[d]& X^0_2\ar[r]\ar[d]& C_{2}\ar[r]\ar[d]&0\\
\cdots \ar[r]& X^n_3\ar[r]\ar[d]& X^{n-1}_3\ar[r]\ar[d]& \cdots\ar[r]
& X^1_3\ar[r]\ar[d]& X^0_3\ar[r]\ar[d]& C_{3}\ar[r]\ar[d]&0\\
&0&0&&0&0&0&\\}\]
satisfying the following conditions.
\begin{itemize}
\item[$(1)$] The top and bottom rows are minimal $\mathcal{X}$-resolutions of $C_1$ and $C_3$ respectively,
and the middle row is an $\mathcal{X}$-resolution of $C_2$.
\item[$(2)$] For any $i\geqslant 1$, set $Y^i:=\Ker(X^{i-1}_2\to X^{i-2}_2)$ (note: $X^{-1}_2=C_2$). Then
$Y^i=\Omega^{i}_{\X}(C_2)\oplus X^i$ for some $X^i\in\mathcal{X}$, and all
$\Omega^{i}_{\X}(C_1)$, $\Omega^{i}_{\X}(C_3)$ and $Y^i$ are in $\mathcal{X}^{\bot}$. Moreover, for any $i\geqslant 1$,
we have the following exact sequence
$$0\to \Omega^{i}_{\X}(C_1) \to Y^i(=\Omega^{i}_{\X}(C_2)\oplus X^i)\to \Omega^{i}_{\X}(C_3)\to 0.\eqno{(3{\text -}i)}$$
In particular, if $\Omega^{n}_{\X}(C_3)\in\mathcal{X}$ for some $n\geqslant 1$, then the sequence $(3{\text -}n)$ splits.
\end{itemize}
\end{lemma}

\begin{proof}
Since $\X$ is a contravariantly finite and resolving subcategory of $\mod\Lambda$,
by \cite[Proposition 3.3(c)]{AuslanderApplications} we have minimal $\mathcal{X}$-resolutions
$$\cdots \to X^n_1\to X^{n-1}_1\to\cdots\to X^1_1\to X^0_1\to C_{1}\to 0,$$
$$\cdots \to X^n_3\to X^{n-1}_3\to\cdots\to X^1_3\to X^0_1\to C_{3}\to 0$$
of $C_1$ and $C_3$ respectively with all $\Omega^{i}_{\X}(C_1)$ and $\Omega^{i}_{\X}(C_3)$  are in $\mathcal{X}^{\bot}$.
Then by \cite[Proposition 3.6]{AuslanderApplications}, we get the commutative diagram as above such that
all $Y^i$ are in $\mathcal{X}^{\bot}$, where $Y^i=\Ker(X^{i-1}_2\to X^{i-2}_2)$ (note: $X^{-1}_2=C_2$). It follows that
the middle row in the above diagram is an $\mathcal{X}$-resolutions of $C_2$ and $Y^i=\Omega^{i}_{\X}(C_2)\oplus X^i$
with $X^i\in\mathcal{X}$ for any $i\geqslant 1$. In particular, we have the following exact sequence
$$0\to \Omega^{i}_{\X}(C_1) \to Y^i(=\Omega^{i}_{\X}(C_2)\oplus X^i)\to \Omega^{i}_{\X}(C_3)\to 0$$
for any $i\geqslant 1$, which induces an exact sequence
$$0\to \Hom_{\Lambda}(\Omega^{i}_{\X}(C_3),\Omega^{i}_{\X}(C_1))\to \Hom_{\Lambda}(\Omega^{i}_{\X}(C_3),Y^{i})
\to \Hom_{\Lambda}(\Omega^{i}_{\X}(C_3),\Omega^{i}_{\X}(C_3))\to\Ext^{1}_{\Lambda}(\Omega^{i}_{\X}(C_3),\Omega^{i}_{\X}(C_1)).$$
If $\Omega^{n}_{\X}(C_3)\in\mathcal{X}$ for some $n\geqslant 1$, then $\Ext^{1}_{\Lambda}(\Omega^{n}_{\X}(C_3),\Omega^{n}_{\X}(C_1))=0$
and the exact sequence $(3{\text -}n)$ splits.
\end{proof}

Let $\mathcal{B}$ be a subclass of $\mod \Lambda$.
If $\X$ is a contravariantly finite and resolving subcategory of $\mod\Lambda$,
then the {\bf $\X$-projective dimension} $\pd_{\X}\mathcal{B}$ of $\mathcal{B}$ is defined as
\begin{equation*}
\pd_{\X} \B=
\begin{cases}
\sup\{\pd_{\X} M\mid M\in \B\}, & \text{if} \;\; \B \neq\varnothing;\\
-1,&\text{if} \;\; \B =\varnothing.
\end{cases}
\end{equation*}
If $\X$ is a covariantly finite and coresolving subcategory of $\mod\Lambda$, then
the {\bf $\X$-injective dimension} $\id_{\X}\mathcal{B}$ of $\mathcal{B}$ is defined as
\begin{equation*}
\id_{\X} \B=
\begin{cases}
\sup\{\id_{\X} M\mid M\in \B\}, & \text{if} \;\; \B \neq\varnothing;\\
-1,&\text{if} \;\; \B =\varnothing.
\end{cases}
\end{equation*}

\begin{lemma}\label{lem-3.11}
Let $\mathcal{V}$ be a subset of $\mathcal{S}$ and $M\in \mathfrak{F}(\mathcal{V})$. Then we have
\begin{itemize}
\item[$(1)$] $\pd_{\X}M\leqslant\pd_{\X}\mathcal{V}$; in particular, $\pd_{\X}q_{t_{\mathcal{V}}}(M)\leqslant\pd_{\X}\mathcal{V}$.
\item[$(2)$] $\id_{\X}M\leqslant\id_{\X}\mathcal{V}$; in particular, $\id_{\X}q_{t_{\mathcal{V}}}(M)\leqslant\id_{\X}\mathcal{V}$.
\end{itemize}
\end{lemma}

\begin{proof}
(1) Let $M\in\mathfrak{F}\,(\mathcal{V})$. Then there exists a finite chain
$$0=M_0\subseteq M_1\subseteq \cdots\subseteq M_m=M$$ of submodules of $M$
such that each quotient $M_i / M_{i-1}$ is isomorphic to some module in $\mathcal{V}$. It follows from Lemma \ref{lem-3.10}(2) that
$\pd_{\X}M\leqslant \pd_{\X}\mathcal{V}$. In particular, $\pd_{\X}q_{t_{\mathcal{V}}}(M)\leqslant \pd_{\X}\mathcal{V}$
since $q_{t_{\mathcal{V}}}(M)\in\mathfrak{F}\,(\mathcal{V})$.

(2) It is dual to (1).
\end{proof}


Recall that a category $\mathcal{X}$ of $\mod\Lambda$ is said to be of {\bf finite type} if there are only finitely many
pairwise non-isomorphic indecomposable modules in $\mathcal{X}$.
Also recall that $\mathcal{S}$ denotes the set of all pairwise non-isomorphic simple modules in $\mod\Lambda$.
We are in a position to prove the following result.

\begin{theorem}\label{thm-3.12}
Let $\mathcal{V}$ be a subset of $\mathcal{S}$ and $\X$ a subcategory of $\mod\Lambda$ of finite type.
\begin{itemize}
\item[$(1)$] If $\X$ is resolving, then
$\extdim\Lambda\leqslant\pd_{\X}\mathcal{V}+\ell\ell^{t_{\mathcal{V}}}(\Lambda)$.
\item[$(2)$] If $\X$ is coresolving, then
$\extdim\Lambda\leqslant\id_{\X}\mathcal{V}+\ell\ell^{t_{\mathcal{V}}}(\Lambda)$.
\end{itemize}
\end{theorem}

\begin{proof}
Set $\ell\ell^{t_{\mathcal{V}}}(\Lambda)=n$. Since $\X$ is of finite type, we have that $\X$ is
contravariantly and covariantly finite and $\X=\add X$ for some $X\in \mod \Lambda$.

If $n=0$, that is, $\ell\ell^{t_{\mathcal{V}}}(\Lambda)=0$, then
$M\cong q_{t_\mathcal{V}}(M)$ by Lemmas \ref{lem-3.9} and \ref{lem-3.8}(1),
and hence $\pd_{\X}M=\pd_{\X}q_{t_{\mathcal{V}}}(M)\leqslant \pd_{\X}\mathcal{V}$ and
$\id_{\X}M=\id_{\X}q_{t_{\mathcal{V}}}(M)\leqslant \id_{\X}\mathcal{V}$ by Lemma \ref{lem-3.11}.
The case for $n=0$ is proved. Now suppose $n\geqslant 1$.

$(1)$ Let $\pd_{\X} \mathcal{V}=p<\infty$. By Lemma \ref{lem-3.11}(1), we have
$\pd_{\X} q_{t_{\mathcal{V}}}F^{i}_{t_{\mathcal{V}}}(M)\leqslant p$ and
$\Omega^{p+1}_{\X} (q_{t_{\mathcal{V}}}F^{i}_{t_{\mathcal{V}}}(M))=0$
for any $0\leqslant i\leqslant n-1$.
By Lemma \ref{lem-3.8}(2), we have $F^{n}_{t_{\mathcal{V}}}(M)\in\mathfrak{F}(\mathcal{V})$.
It follows from Lemma \ref{lem-3.11}(1) that
$\pd_{\X} F^{n}_{t_{\mathcal{V}}}(M) \leqslant p$. Thus $\Omega^{p}_{\X}(F^{n}_{t_{\mathcal{V}}}(M))\in\X$
and $\Omega^{p+1}_{\X}(F^{n}_{t_{\mathcal{V}}}(M))=0$.

By Lemmas \ref{lem-3.7} and \ref{lem-3.10}, we have
\begin{align}\label{iso1}
\Omega^{p+1}_{\X}(t_{\mathcal{V}}F^{i}_{t_{\mathcal{V}}}(M))\cong\Omega^{p+1}_{\X}(F^{i}_{t_{\mathcal{V}}}(M))\oplus X_{i},
\end{align}
\begin{align}\label{iso2}
0 \rightarrow \Omega^{p+1}_{\X}(F^{i+1}_{t_{\mathcal{V}}}(M))\rightarrow
\Omega^{p+1}_{\X}(t_{\mathcal{V}}F^{i}_{t_{\mathcal{V}}}(M))\oplus X'_{i}
\rightarrow \Omega^{p+1}_{\X}(\top t_{\mathcal{V}}F^{i}_{t_{\mathcal{V}}}(M))\rightarrow 0\ \ {\rm (exact)}
\end{align}
with $X_{i},X'_{i}\in \add X=\X$ for any $0\leqslant i\leqslant n-1$.
In particular, when $i=n-1$ in (\ref{iso2}), we have
\begin{align}\label{iso5}
\Omega^{p+1}_{\X}(t_{\mathcal{V}}F^{n-1}_{t_{\mathcal{V}}}(M))\oplus X'_{n-1}\cong
\Omega^{p+1}_{\X}(\top t_{\mathcal{V}}F^{n-1}_{t_{\mathcal{V}}}(M)).
\end{align}
It follows that
\begin{align*}
[\Omega^{p+1}_{\X}(M)]_{1}\subseteq&[\Omega^{p+1}_{\X}(t_{\mathcal{V}}(M))]_{1} \ \ \ \ \text{(putting $i=0$ in (\ref{iso1}))}\\
\subseteq &[\Omega^{p+1}_{\X}(F_{t_{\mathcal{V}}}(M))]_{1}\bullet[\Omega^{p+1}_{\X}(\top t_{\mathcal{V}}(M))]_{1}
\ \ \ \ \text{(putting $i=0$ in (\ref{iso2}))}\\
\subseteq &[\Omega^{p+1}_{\X}(t_{\mathcal{V}}F_{t_{\mathcal{V}}}(M))]_{1}\bullet[\Omega^{p+1}_{\X}(\top t_{\mathcal{V}}(M))]_{1}
\ \ \ \ \text{(putting $i=1$ in (\ref{iso1}))}\\
\subseteq&[\Omega^{p+1}_{\X}(t_{\mathcal{V}}F_{t_{\mathcal{V}}}(M))]_{1}\bullet[\Omega^{p+1}_{\X}(\Lambda/\rad\Lambda)]_{1}.
\end{align*}
By replacing $M$ with $F^{i}_{t_{\mathcal{V}}}(M)$ for any $1 \leqslant i \leqslant n-1$ and iterating the above process, we have
\begin{align*}
[\Omega^{p+1}_{\X}(M)]_{1}
\subseteq &[\Omega^{p+1}_{\X}(t_{\mathcal{V}}F^{n-1}_{t_{\mathcal{V}}}(M))]_{1} \bullet
[\Omega^{p+1}_{\X}(\Lambda/\rad\Lambda)]_{n-1} \ \ \ \ \\ 
\subseteq &[\Omega^{p+1}_{\X}(\top t_{\mathcal{V}}F^{n-1}_{t_{\mathcal{V}}}(M))]_{1} \bullet
[\Omega^{p+1}_{\X}(\Lambda/\rad\Lambda)]_{n-1} \ \ \ \ \text{(by (\ref{iso5}))}\\
\subseteq &[\Omega^{p+1}_{\X}(\Lambda/\rad\Lambda)]_{1} \bullet [\Omega^{p+1}_{\X}(\Lambda/\rad\Lambda)]_{n-1} \ \ \ \ \\
\subseteq &[\Omega^{p+1}_{\X}(\Lambda/\rad\Lambda)]_{n}. \ \ \ \ \text{(by Lemma \ref{lem-2.2})}
\ \ \ \ \ \ \ \ \ \ \ \ \ \ \ \ \ \ \ \ \ \ \
\ \ \ \ \ \ \ \ \ \ \ \ \ \ \ \ \ \ \ \ \ \ \ \ \ \ \ \ \ \ \ \ \ \ \ \ \ (3.4)
\end{align*}

Consider the following exact sequence
$$0 \longrightarrow\Omega^{p+1}_{\X}(M) \longrightarrow X_{ p }\longrightarrow X_{p-1}
\longrightarrow\cdots \longrightarrow X_{1} \longrightarrow X_{0} \longrightarrow M \longrightarrow 0$$
in $\mod \La$ with all $X_{i}$ in $\add X=\X$.
Thus we have
\begin{align*}
[M]_{1}\subseteq & [X_{0}]_{1}\bullet[\Omega^{-1}(X_{1})]_{1}\bullet\cdots\bullet[\Omega^{-p}(X_{p})]_{1}\bullet
[\Omega^{-(p+1)}(\Omega^{p+1}_{\X}(M))]_{1} \ \ \ \ \text{(by Lemma \ref{lem-3.5}(2))}\\
\subseteq &[\oplus_{i=0}^{p} \Omega^{-i}(X)]_{p+1}\bullet[\Omega^{-(p+1)}(\Omega^{p+1}_{\X}(M))]_{1} \ \ \ \ \text{(by Lemma \ref{lem-2.2})}\\
\subseteq&[\oplus_{i=0}^{p}\Omega^{-i}(X)]_{p+1}\bullet[\Omega^{-(p +1)}(\Omega^{p+1}_{\X}(\Lambda/\rad \Lambda))]_{n}
\ \ \ \ \ \text{(by (3.4) and Lemma \ref{lem-3.6}(1)) }\\
\subseteq&[\oplus_{i=0}^{p} \Omega^{-i}(X)\oplus \Omega^{-(p+1)}(\Omega^{p+1}_{\X}(\Lambda/\rad \Lambda))]_{p +1+n}.
\ \ \ \ \text{(by Lemma \ref{lem-2.2})}
\end{align*}
It follows that
$$\mod\Lambda=[\oplus_{i=0}^{p}\Omega^{-i}(X) \oplus \Omega^{-(p+1)}(\Omega^{p+1}_{\X}(\Lambda/\rad \Lambda))]_{p+1+n}$$
and $\extdim\Lambda\leqslant p+n$.

$(2)$ The proof is dual to that of (1), but we still give it here for the readers' convenience.

Let $\id_{\X} \mathcal{V}=p<\infty$. By Lemma \ref{lem-3.11}(2), we have
$\id_{\X} q_{t_{\mathcal{V}}}F^{i}_{t_{\mathcal{V}}}(M)\leqslant p$ and
$\Omega^{-(p+1)}_{\X} (q_{t_{\mathcal{V}}}F^{i}_{t_{\mathcal{V}}}(M))=0$
for any $0\leqslant i\leqslant n-1$.
By Lemma \ref{lem-3.8}(2), we have $F^{n}_{t_{\mathcal{V}}}(M)\in\mathfrak{F}(\mathcal{V})$. Then by Lemma \ref{lem-3.11}(2),
we have $\id_{\X} F^{n}_{t_{\mathcal{V}}}(M) \leqslant p$. Thus $\Omega^{-p}_{\X}(F^{n}_{t_{\mathcal{V}}}(M))\in\X$
and $\Omega^{-(p+1)}_{\X}(F^{n}_{t_{\mathcal{V}}}(M))=0$.

By Lemma \ref{lem-3.7} and the dual of Lemma \ref{lem-3.10}, we have
$$\Omega^{-(p+1)}_{\X}(t_{\mathcal{V}}F^{i}_{t_{\mathcal{V}}}(M))\cong
\Omega^{-(p+1)}_{\X}(F^{i}_{t_{\mathcal{V}}}(M))\oplus X_{i},\eqno{(3.5)}$$
$$0 \rightarrow \Omega^{-(p+1)}_{\X}(F^{i+1}_{t_{\mathcal{V}}}(M))\rightarrow
\Omega^{-(p+1)}_{\X}(t_{\mathcal{V}}F^{i}_{t_{\mathcal{V}}}(M))\oplus X'_{i}
\rightarrow \Omega^{-(p+1)}_{\X}(\top t_{\mathcal{V}}F^{i}_{t_{\mathcal{V}}}(M))\rightarrow 0\ \ {\rm (exact)}\eqno{(3.6)}$$
with $X_{i},X'_{i}\in \add X=\X$ for any $0\leqslant i\leqslant n-1$.
In particular, when $i=n-1$ in (3.6), we have
$$\Omega^{-(p+1)}_{\X}(t_{\mathcal{V}}F^{n-1}_{t_{\mathcal{V}}}(M))\oplus X'_{n-1}\cong
\Omega^{-(p+1)}_{\X}(\top t_{\mathcal{V}}F^{n-1}_{t_{\mathcal{V}}}(M)).\eqno{(3.7)}$$
It follows that
\begin{align*}
[\Omega^{-(p+1)}_{\X}(M)]_{1}\subseteq&[\Omega^{-(p+1)}_{\X}(t_{\mathcal{V}}(M))]_{1} \ \ \ \ \text{(putting $i=0$ in (3.5))}\\
\subseteq &[\Omega^{-(p+1)}_{\X}(F_{t_{\mathcal{V}}}(M))]_{1}\bullet[\Omega^{-(p+1)}_{\X}(\top t_{\mathcal{V}}(M))]_{1}
\ \ \ \ \text{(putting $i=0$ in (3.6))}\\
\subseteq &[\Omega^{-(p+1)}_{\X}(t_{\mathcal{V}}F_{t_{\mathcal{V}}}(M))]_{1}\bullet[\Omega^{-(p+1)}_{\X}(\top t_{\mathcal{V}}(M))]_{1}
\ \ \ \ \text{(putting $i=1$ in (3.5))}\\
\subseteq&[\Omega^{-(p+1)}_{\X}(t_{\mathcal{V}}F_{t_{\mathcal{V}}}(M))]_{1}\bullet[\Omega^{-(p+1)}_{\X}(\Lambda/\rad\Lambda)]_{1}.
\end{align*}
By replacing $M$ with $F^{i}_{t_{\mathcal{V}}}(M)$ for any $1 \leqslant i \leqslant n-1$ and iterating the above process, we have
\begin{align*}
[\Omega^{-(p+1)}_{\X}(M)]_{1}
\subseteq &[\Omega^{-(p+1)}_{\X}(t_{\mathcal{V}}F^{n-1}_{t_{\mathcal{V}}}(M))]_{1} \bullet
[\Omega^{-(p+1)}_{\X}(\Lambda/\rad\Lambda)]_{n-1} \ \ \ \ \\ 
\subseteq &[\Omega^{-(p+1)}_{\X}(\top t_{\mathcal{V}}F^{n-1}_{t_{\mathcal{V}}}(M))]_{1} \bullet
[\Omega^{-(p+1)}_{\X}(\Lambda/\rad\Lambda)]_{n-1} \ \ \ \ \text{(by (3.7))}\\
\subseteq &[\Omega^{-(p+1)}_{\X}(\Lambda/\rad\Lambda)]_{1} \bullet [\Omega^{-(p+1)}_{\X}(\Lambda/\rad\Lambda)]_{n-1} \ \ \ \ \\
\subseteq &[\Omega^{-(p+1)}_{\X}(\Lambda/\rad\Lambda)]_{n}. \ \ \ \ \text{(by Lemma \ref{lem-2.2})}
\ \ \ \ \ \ \ \ \ \ \ \ \ \ \ \ \ \ \ \ \ \ \
\ \ \ \ \ \ \ \ \ \ \ \ \ \ \ \ \ \ \ \ \ \ \ \ \ \ \ \ \ \ \ \ \ \ \ \ \ (3.8)
\end{align*}

Consider the following exact sequence
$$0\to M\to X^0\to X^1\to\cdots\to X^p\to\Omega^{-(p+1)}_{\X}(M)\to 0$$
in $\mod \La$ with all $X^{i}$ in $\add X=\X$.
Thus we have
\begin{align*}
[M]_{1}\subseteq & [X^{0}]_{1}\bullet[\Omega^{1}(X^{1})]_{1}\bullet\cdots\bullet[\Omega^{p}(X^{p})]_{1}\bullet
[\Omega^{p+1}(\Omega^{-(p+1)}_{\X}(M))]_{1} \ \ \ \ \text{(by Lemma \ref{lem-3.5}(1))}\\
\subseteq &[\oplus_{i=0}^{p} \Omega^{i}(X)]_{p+1}\bullet[\Omega^{p+1}(\Omega^{-(p+1)}_{\X}(M))]_{1}\\
\subseteq&[\oplus_{i=0}^{p}\Omega^{i}(X)]_{p+1}\bullet[\Omega^{p+1}(\Omega^{-(p+1)}_{\X}(\Lambda/\rad \Lambda))]_{n}
\ \ \ \ \ \text{(by (3.8) and Lemma \ref{lem-3.6}(2)) }\\
\subseteq&[\oplus_{i=0}^{p} \Omega^{i}(X)\oplus \Omega^{p+1}(\Omega^{-(p+1)}_{\X}(\Lambda/\rad \Lambda))]_{p+1+n}.
\ \ \ \ \text{(by Lemma \ref{lem-2.2})}
\end{align*}
It follows that
$$\mod\Lambda=[\oplus_{i=0}^{p}\Omega^{i}(X) \oplus \Omega^{p+1}(\Omega^{-(p+1)}_{\X}(\Lambda/\rad \Lambda))]_{p+1+n}$$
and $\extdim\Lambda\leqslant p+n$.
\end{proof}

By using exactly the same method, it can be proved that Theorem \ref{thm-3.12} holds true
in the following more general case.

\begin{remark}{\rm \label{rem-3.13}
Let $(\mathcal{T},\mathcal{F})$ be a torsion pair in $\mod \Lambda$ and $t$ its torsion radical,
and let $\mathcal{X}$ be a subcategory of $\mod\Lambda$ of finite type.
\begin{itemize}
\item[$(1)$] If $\mathcal{X}$ is resolving, then
$\extdim\Lambda\leqslant\pd_{\mathcal{X}}\mathcal{F}+\ell\ell^{t}(\Lambda)$.
\item[$(2)$] If $\mathcal{X}$ is coresolving, then
$\extdim\Lambda\leqslant\id_{\mathcal{X}}\mathcal{F}+\ell\ell^{t}(\Lambda)$.
\end{itemize}}
\end{remark}

\subsection{Some applications}

\begin{corollary}\label{cor-3.14}
Let $\X$ be a subcategory of $\mod\Lambda$ of finite type.
\begin{itemize}
\item[$(1)$] If $\X$ is resolving, then $\extdim\Lambda\leqslant\pd_{\X}\mathcal{S}$.
\item[$(2)$] If $\X$ is coresolving, then $\extdim\Lambda\leqslant\id_{\X}\mathcal{S}$.
\end{itemize}
\end{corollary}

\begin{proof}
It follows from Theorem \ref{thm-3.12} and Lemma \ref{lem-3.9}.
\end{proof}

If $\mathcal{X}$ is the subcategory of $\mod\Lambda$ consisting of projective (resp. injective) modules, then
the $\mathcal{X}$-projective dimension $\pd_\mathcal{X}M$ (resp. $\mathcal{X}$-injective dimension $\id_\mathcal{X}M$)
of a module $M$ in $\mod\Lambda$ is exactly its projective dimension $\pd M$ (resp. injective dimension $\id M$).
In this case, for a subclass of $\mod\Lambda$, we write
$$\pd \mathcal{B}:=\pd_\mathcal{X}\mathcal{B}\ {\rm and}\ \id \mathcal{B}:=\id_\mathcal{X}\mathcal{B}.$$

\begin{corollary}\label{cor-3.15}
\begin{itemize}
\item[]
\item[$(1)$] $\derdim\Lambda \leqslant 2\extdim\Lambda+1$.
\item[$(2)$] For any subset $\mathcal{V}$ of $\mathcal{S}$, we have
\begin{itemize}
\item[$(2.1)$] $\extdim\Lambda \leqslant \min\{\pd\mathcal{V}, \id\mathcal{V}\}+\ell\ell^{t_{\mathcal{V}}}(\Lambda)$.
\item[$(2.2)$] $\derdim\Lambda \leqslant 2(\min\{\pd\mathcal{V}, \id\mathcal{V}\}+\ell\ell^{t_{\mathcal{V}}}(\Lambda))+1$.
\end{itemize}
\item[$(3)$] {\rm (\cite[4.5.1(3)]{iyama2003rejective})} $\extdim\Lambda \leqslant\gldim\Lambda$.
\end{itemize}
\end{corollary}

\begin{proof}
The assertion (1) is a direct consequence of Theorem \ref{thm-3.3}. The assertion (2.1) follows from Theorem \ref{thm-3.12},
and (2.2) follows from (1) and (2.1). Since $\gldim\Lambda=\pd\mathcal{S}=\id\mathcal{S}$, the assertion (3) is a special case
of Corollary \ref{cor-3.14}.
\end{proof}

\begin{corollary}\label{cor-3.16} {\rm (\cite[Theorem]{han2009derived})}
If $\Lambda$ is of finite representation type, then $\derdim\Lambda\leqslant 1$.
\end{corollary}

\begin{proof}
It is easy to see that $\Lambda$ is of finite representation type if and only if $\extdim\Lambda=0$
(\cite[Example 1.6(i)]{beligiannis2008some}). Now the assertion follows from Corollary \ref{cor-3.15}(1).
\end{proof}

For any $n\geqslant 0$, recall from \cite{WJQ09F} that $\Lambda$
is called {\bf $n$-Igusa-Todorov} if there exists $U\in\mod \Lambda$
such that for any $M\in\mod \Lambda$, there exists an exact sequence
$$0 \rightarrow U_{1}\longrightarrow U_{0}\rightarrow \Omega^{n}(M)\oplus P\rightarrow 0$$
in $\mod \Lambda$ with $U_{1}$, $U_{0}\in\add U$ and $P$ projective.
The class of Igusa-Todorov algebras includes algebras with representation dimension
at most 3, algebras with radical cube zero, monomial algebras, left serial algebras and syzygy finite algebras
(\cite{WJQ09F}).

\begin{corollary}\label{cor-3.17}
\begin{itemize}
\item[]
\item[$(1)$]
If $\Lambda$ is an $n$-Igusa-Todorov algebra, then $\derdim\Lambda\leqslant 2n+3$.
\item[$(2)$] $\derdim\Lambda\leqslant 5$ if $\Lambda$ is one class of the following algebras.
\begin{itemize}
\item[$(2.1)$] monomial algebras;
\item[$(2.2)$] left serial algebras;
\item[$(2.3)$] $\rad^{2n+1}\Lambda=0$ and $\Lambda/\rad^n\Lambda$ is of finite representation type;
\item[$(2.4)$] 2-syzygy finite algebras.
\end{itemize}
\end{itemize}
\end{corollary}

\begin{proof}
(1) If $\Lambda$ is $n$-Igusa-Todorov, then $\extdim\Lambda\leqslant n+1$ by \cite[Proposition 3.15(2)]{zheng2019extension}.
Thus $\derdim\Lambda\leqslant 2n+3$ by Corollary \ref{cor-3.15}(1).

(2) The assertion follows from \cite[Corollary 3.16]{zheng2019extension} and Corollary \ref{cor-3.15}(1).
\end{proof}


Set $$u_{1}:=2(\min\{\pd\mathcal{V}, \id\mathcal{V}\}+\ell\ell^{t_{\mathcal{V}}}(\Lambda))+1,$$
$$u_{2}:=(\min\{\pd\mathcal{V}, \id\mathcal{V}\}+2)(\ell\ell^{t_{\mathcal{V}}}(\Lambda)+1)-2.$$
Then $u_{2}-u_{1}=(\min\{\pd\mathcal{V}, \id\mathcal{V}\})(\ell\ell^{t_{\mathcal{V}}}(\Lambda)-1)-1$.
Thus $u_{2}-u_{1}\geqslant 0$ if and only if $\min\{\pd\mathcal{V},\id\mathcal{V}\}\geqslant 1$ and
$\ell\ell^{t_{\mathcal{V}}}(\Lambda)\geqslant 2$. Now, combining Corollary \ref{cor-3.15}(2.2) with
\cite[Theorem 3.12]{zheng2017upper}, we get our main result as follows.

\begin{theorem}\label{thm-3.18}
Let $\mathcal{V}$ be a subset of $\mathcal{S}$, and let $\min\{\pd\mathcal{V}, \id\mathcal{V}\}=d$
and $\ell\ell^{t_{\mathcal{V}}}(\Lambda)=n$. Then we have
$$\derdim\Lambda\leqslant
\begin{cases}
2(d+n)+1, &\mbox{if $d\geqslant 1$ and $n\geqslant 2$;}\\
(d+2)(n+1)-2, &\mbox{otherwise.}
\end{cases}$$
\end{theorem}

Now we compare the upper bounds obtained in the above theorem with those known
upper bounds for $\derdim\Lambda$.

\begin{remark}\label{rem-3.19}
{\rm Keeping the notation as above, the following results have been known.
\begin{itemize}
\item[$(1)$] $\derdim\Lambda \leqslant \LL(\Lambda)-1$ (\cite[Proposition 7.37]{rouquier2006representation});
\item[$(2)$] $\derdim \Lambda \leqslant \gldim \Lambda$
(\cite[Proposition 7.4]{rouquier2006representation} and \cite[Proposition 2.6]{KrK});
\item[$(3)$] $\derdim\Lambda\leqslant (d+2)(n+1)-2$ (\cite[Theorem 3.12]{zheng2017upper}).
\end{itemize}
According to the argument before Theorem \ref{thm-3.18}, we have that when
$d\geqslant 1$ and $n\geqslant 2$, the upper bounds in Theorem \ref{thm-3.18} are
at most that in (3), with equality if $d=1$ and $n=2$; otherwise, they coincide.

It was pointed out in \cite[Remark 3.16]{zheng2017upper} that
if $\mathcal{V}=\varnothing$, then the upper bounds in (1) and (3) coincide; and if $\mathcal{V}=\mathcal{S}$,
then the upper bounds in (2) and (3) coincide. By choosing suitable $\mathcal{V}$, the upper bounds in (3) are
smaller than that in (1) and (2) and the difference may be arbitrarily large;
see \cite[Examples 4.1 and 4.2]{zheng2017upper}.}
\end{remark}




\section{Examples}

In this section, we give some examples to illustrate our results.

\begin{example} \label{exa-4.1}
{\rm Let $k$ be an algebraically closed field and $\Lambda=kQ/I$, where $Q$ is the quiver
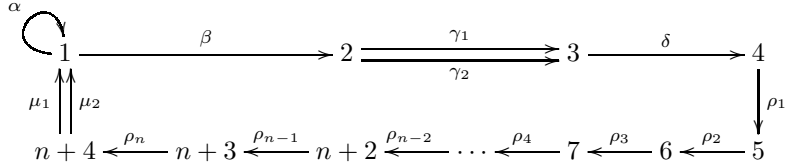
\[\xymatrix{
&{1}\ar@(l,u)^{\alpha} \ar[rr]^{\beta} &&2 \ar@<0.5ex>[rr]^{\gamma_{1}}\ar@<-0.5ex>[rr]_{\gamma_{2}}
&& 3 \ar[rr]^{\delta}&&4 \ar[d]^{\rho_{1}}&\\
&n+4\ar@<0.5ex>[u]^{\mu_{1}} \ar@<-0.5ex>[u]_{\mu_{2}} &n+3\ar[l]_{\rho_{n}}&n+2\ar[l]_{\rho_{n-1}}
&\cdots\ar[l]_{\rho_{n-2}}&7\ar[l]_{\rho_{4}}&6\ar[l]_{\rho_{3}}&5\ar[l]_{\rho_{2}}
}\]
and $I$ is generated by $\{\alpha^{m},\alpha\beta,
\gamma_{1}\delta-\gamma_{2}\delta, \rho_{n}\mu_{1}\alpha, \rho_{n}\mu_{2}\alpha, \mu_{1}\beta-\mu_{2}\beta\}$
with $m \geqslant 4$ and $n\geqslant 1$ (note: following \cite{ASS,SY}, we concatenate the arrows from
left to right). Then the indecomposable projective $\Lambda$-modules are
$$
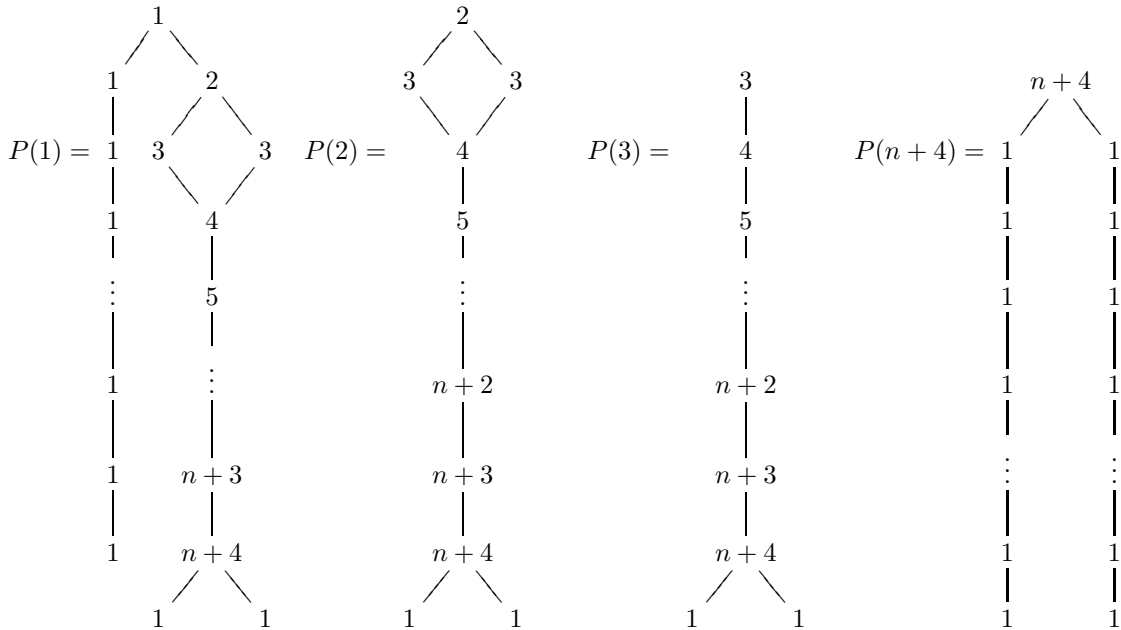
\xymatrix@-1.0pc@C=0.1pt
{&  &  &{1}\edge[lld]\edge[dr]&& &&&&{2}\edge[ld]\edge[rd]&& & &&&&  &&&&  & &  \\
&{1}\edge[d] &&&{2}\edge[ld]\edge[rd]&& &&3\edge[rd]&& 3\edge[ld]&&& &&&3\edge[d]
&&&&&&{n+4}\edge[ld]\edge[rd]&&& \\
P(1)=&{1}\edge[d] &&3\edge[rd] &&3\edge[ld]& &P(2)=&&4\edge[d] && &&&P(3)=& &4\edge[d]
&& & & P(n+4)=&{1}\edge[d]  &  &{1} \edge[d] ,  \\
&{1}\edge[d] && &4\edge[d]&&&&&5\edge[d]&&&&&& &5\edge[d]&&& & &{1}\edge[d]
&  &{1} \edge[d] & &&& &  & \\
&\vdots\edge[d]  &&&5\edge[d]&& &&&\vdots\edge[d]&&&& &&&\vdots\edge[d]&&&&&{1}\edge[d]
&  &{1} \edge[d]& &&&&&&& &  & \\
&{1}\edge[d] &&&\vdots\edge[d] &&&&&n+2\edge[d]  & &&&& & &n+2\edge[d]&&& &&{1}\edge[d]
&  &{1} \edge[d]&&&&& &  &&&& &  &  \\
     &{1}\edge[d] &&&{n+3}\edge[d]   &&&&& {n+3}\edge[d] & && && & &{n+3}\edge[d]
     &&& &&\vdots\edge[d]  &  &\vdots \edge[d]&&&&& &  & &&& &  & \\
     &{1} &&&{n+4}\edge[ld]\edge[rd]  &&&&&{n+4}\edge[ld]\edge[rd]
     & &&&& & & {n+4}\edge[ld]\edge[rd]&&& &&{1}\edge[d]  &  &{1} \edge[d]&&&&& &  & &&& &  & \\
     &&  &{1}&&{1}&&&{1}&  &{1}&&&& &{1}&&{1}&& &&{1}  &  &{1} &&&&& &  & \\
}$$
and $P(i)=\rad P(i-1)$ for each $4\leqslant i\leqslant n+3$. It is straightforward to verify that
\begin{equation*}
\pd S(i)=
\begin{cases}
\infty, &\text{if}\;\; i=1, n+3;\\
2,&\text{if} \;\;i=2,n+4;\\
1,&\text{if} \;\; 3 \leqslant  i\leqslant n+2.
\end{cases}
\end{equation*}
Let $\mathcal{V}:=\{S(i)\mid 3 \leqslant  i\leqslant n+2\}$. Then
$\pd\mathcal{V}=1$.
Let $\mathcal{V}'$ be all the others simple modules in $\mod \Lambda$, that is,
$\mathcal{V}'=\{ S(1),S(2), S(n+3),S(n+4)\}$. Since $\La=\oplus_{i=1}^{n+4}P(i)$, we have
$$\ell\ell^{t_{\V}}(\Lambda)=\max\{\ell\ell^{t_{\V}}(P(i)) \mid 1 \leqslant i  \leqslant n+4\}$$
by \cite[Lemma 3.4(a)]{huard2013layer}.

In order to compute $\ell\ell^{t_{\V}}(P(1))$, we need to find the least non-negative integer $i$
such that $t_{\V}F_{t_{\V}}^{i}(P(1))=0$.
Since $\top P(1)=S(1)\in \add \V'$, we have $t_{\V}(P(1))=P(1)$ by \cite[Proposition 5.9(a)]{huard2013layer}.
Thus
\begin{align*}\xymatrix@-1.0pc@C=0.1pt{
&F_{t_{\V}}(P(1))=\rad t_{\V}(P(1))=\rad P(1)=T_{m-1}\oplus P(2),\\
}\end{align*}
\xymatrix@-1.0pc@C=0.05pt {
& {1}\edge[d]             \\
{\rm where}  \;\;\;T_{m-1}=  &{1}\edge[d]&({\rm the \;number\; of\; 1 \;is \;}m-1).\;\\
&\vdots\edge[d] \\
& {1}}

Since $\top T_{m-1}=S(1)\in \V'$, we have $t_{\V}(T_{m-1})=T_{m-1}$ by \cite[Proposition 5.9(a)]{huard2013layer} again.
Similarly, $t_{\V}(P(2))=P(2)$. Thus we have
\begin{align*}\xymatrix@-1.0pc@C=0.1pt{
&t_{\V}F_{t_{\V}}(P(1))= t_{\V}(T_{m-1}\oplus P(2))= t_{\V}(T_{m-1})\oplus t_{\V}( P(2))=T_{m-1}\oplus P(2),
}\end{align*}
\begin{align*}\xymatrix@-1.0pc@C=0.1pt{
&F^{2}_{t_{\V}}(P(1))&=\rad t_{\V}F_{t_{\V}}(P(1))=\rad(T_{m-1}\oplus P(2))=\rad(T_{m-1})\oplus \rad(P(2))=T_{m-2}\oplus M,%
}\end{align*}
\xymatrix@-1.0pc@C=0.05pt {
&3\edge[rd]&&3\edge[ld]\\
&& 4\edge[d]  \\
{\rm where}  \;\;\;M= && 5\edge[d]  \\
&& \vdots\edge[d]  \\
&& n+2\edge[d]  \\
&& {n+3}\edge[d]  \\
&& {n+4}\edge[ld]\edge[rd]  \\
&{1}&&{1}.}

Thus \begin{align*}\xymatrix@-1.0pc@C=0.1pt{
&t_{\V}F^{2}_{t_{\V}}(P(1))&=t_{\V}(T_{m-2}\oplus M)=t_{\V}(T_{m-2})\oplus t_{\V}(M)=T_{m-2}\oplus P(n+3)
.
}\end{align*}
Repeating the process, we get that $S(1)$ is a direct summand of $t_{\V}F^{m-1}_{t_{\V}}(P(1))$, that is, $t_{\V}F^{m-1}_{t_{\V}}(P(1))\neq 0$
and
$t_{\V}F^{m}_{t_{\V}}(P(1))=0.$ It follows that $\ell\ell^{t_{\V}}(P(1))=m$.
Similarly, we have
\begin{equation*}
\ell\ell^{t_{\V}}(P(i))=
\begin{cases}
4,&\text{if} \;\;i=2;\\
3,&\text{if} \;\;3\leqslant  i\leqslant  n+3;\\
m+1,&\text{if} \;\;i= n+4.
\end{cases}
\end{equation*}
Consequently, we conclude that
$$\ell\ell^{t_{\V}}(\La)=\max\{\ell\ell^{t_{\V}}(P(i))\mid 1 \leqslant i  \leqslant  n+4\}=m+1.$$
\begin{itemize}
\item[$(1)$] Since  $\LL(\Lambda)=n+5$ and $\gldim \Lambda=\infty$, by \cite[Example 1.6(ii)]{beligiannis2008some} we have
$$\extdim\Lambda \leqslant \LL(\Lambda)-1=\max\{m-1,n+5\}-1=\max\{m-2,n+4\}.$$
\item[$(2)$] By Corollary \ref{cor-3.15}(2.1), we have
$$\extdim \Lambda \leqslant \pd\mathcal{V}+\ell\ell^{t_{\mathcal{V}}}(\Lambda_{\Lambda})=1+(1+m)=m+2.$$
\item[$(3)$] By \cite[Theorem 3.12]{zheng2017upper}, we have
$$\derdim\Lambda \leqslant(\pd\mathcal{V} +2)(\ell\ell^{t_{\mathcal{V}}}(\Lambda)+1)-2=(1+2)(m+1+1)-2=3m+4.$$
\item[$(4)$] By Corollary \ref{cor-3.15}(2.2), we have
$$\derdim \Lambda\leqslant 2(\pd\mathcal{V} +\ell\ell^{t_{\mathcal{V}}}(\Lambda))+1=2\times(2+m)+1=2m+5.$$
\end{itemize}
Thus, it is clear that by choosing suitable $m$ and $n$, the upper bounds obtained in this paper are more precise,
even arbitrarily smaller, than that in the literature known so far.}
\end{example}


The following example shows that we may obtain the exact value of the derived dimension of some certain algebras.

\begin{example}\label{exa-4.2}
{\rm Let $k$ be an algebraically closed field and $\Lambda=kQ/I$, where $Q$ is the quiver
$$
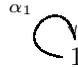
\xymatrix{
&1 \ar@(l,u)^{\alpha_{1}}
}$$
and $I$ is generated by $\{\alpha_{1}^{r}\}$ with $r\geqslant 2$.
Then the indecomposable projective (also injective) $\La$-module is
$$\xymatrix@-1.0pc@C=0.1pt
{
  &  &&&1\edge[d]\\
&&& &1\edge[d] \\
  &&&P(1)=I(1)=&1\edge[d] \\
&&& &\vdots\edge[d]\\
&&& &1.\\
&&&  & \\
}$$
It is verified directly that the injective and projective dimensions of $S(1)$ are infinite and
$\gldim \Lambda=\infty$, and that
$\Lambda$ is a self-injective algebra of finite CM-type. By \cite[Corollary 3.7]{zheng2019extension}, we have $\extdim\Lambda=0$.
It follows from Corollary \ref{cor-3.15}(1) that $\derdim\Lambda\leqslant 1$.
On the other hand, $\derdim\Lambda\geqslant 1$ by \cite{chen2008algebras}.
Thus we conclude that $\derdim\Lambda=1$.}
\end{example}

\vspace{0.5cm}

{\bf Acknowledgements.} This work was partially supported by NSFC (Grant Nos. 11971225, 12001508, 12171207).
The authors thank Hanpeng Gao for discussion of Remark \ref{rem-3.13}, and also thank the referee for useful suggestions.

\end{document}